\documentclass[reqno]{amsart}
\usepackage{mathrsfs}
\usepackage{amsmath}
\usepackage{amsthm}
\usepackage{amsfonts}
\usepackage{amssymb}
\usepackage{url}
\usepackage{enumerate}
\usepackage[pdftex,bookmarks=true]{hyperref}
\usepackage{xcolor}

\newcommand{\corE}{\textcolor{violet}}

\newcommand\rank{{\operatorname{rank}}}

\newcommand\Q{{\mathbf{Q}}}

\renewcommand\P{{\mathbf{P}}}
\newcommand\E{{\mathbf{E}}}

\newcommand\tr{{\operatorname{tr}}}

\newcommand\Z{{\mathbf{Z}}}

\newcommand\F{{\mathbf{F}}}

\newcommand\ep{\varepsilon}

\newcommand\Be{{\mathbf e}}

\newcommand\Bx{{\mathbf x}}

\newcommand\CE{{\mathcal E}}
\newcommand\CF{{\mathcal F}}

\newcommand\CP{{\mathcal P}}

\newcommand\CR{{\mathcal R}}

\renewcommand\mod{\ \operatorname{mod}\ }

\newcommand\supp{\mathbf{supp}}

\newcommand\eps{\varepsilon}

\newcommand\codim{\operatorname{codim}}
\newcommand\lang{\langle}
\newcommand\rang{\rangle}
\newcommand\wh{\widehat}
\newcommand\Sp{\operatorname{Span}}
\renewcommand\a{\alpha}
\newcommand\Spec{\mathbf{Spec}}
\newcommand\cok{\mathbf{Cok}}

\newcommand\bs{\backslash}

\usepackage{fullpage}

\theoremstyle{plain}
\newtheorem{theorem}[subsection]{Theorem}
\newtheorem{conjecture}[subsection]{Conjecture}

\newtheorem{prop}[subsection]{Proposition}

\newtheorem{lemma}[subsection]{Lemma}
\newtheorem{corollary}[subsection]{Corollary}
\newtheorem{cor}[subsection]{Corollary}

\newtheorem{remark}[subsection]{Remark}

\newtheorem{claim}[subsection]{Claim}

\theoremstyle{definition}
\newtheorem{definition}[subsection]{Definition}

\begin{document}
\title{Surjectivity of near square random matrices}
\author{Hoi H. Nguyen}
\email{nguyen.1261@math.osu.edu}
\thanks{The first author is supported by research grant DMS-1600782}

\author{Elliot Paquette}
\email{paquette.30@math.osu.edu}

\email{}


\address{Department of Mathematics, The Ohio State University, 231 West 18th Avenue, Columbus, OH 43210}

\subjclass[2000]{}

    \newcommand{\hoi}[1]{{\color{red} \sf $\clubsuit\clubsuit\clubsuit$ Hoi: [#1]}}

\begin{abstract} We show that a nearly square iid random integral matrix is surjective over the integral lattice with very high probability. This answers a question by Koplewitz \cite{S1}. Our result extends to sparse matrices as well as to matrices of dependent entries.  
\end{abstract}

\maketitle



\section{Introduction}\label{secton:intro} In this note we study random rectangular matrices $M_{n \times (n+u)}=(M_{ij})$ of size $n \times (n+u)$, where $n \to \infty$ and $u\ge 0$, and the entries $M_{ij}$ are i.i.d copies of a random variable $\xi$ taking integral values and such that for any prime $p$
\begin{equation}\label{eqn:xi}
  \max_{x \in \Z/p\Z} \P(\xi=x) \le 1-\alpha_n,
\end{equation}
where $\alpha_n>0$ is a parameter allowed to depend on $n$. Such distributions are called {\it $\alpha_n$-balanced}.

For random square matrices of random Bernoulli entries taking values 0 and 1 with probability 1/2, the problem to estimate the probability $p_n$ of $M_{n \times n}$ being singular has attracted quite a lot of attention. In the early 60's Koml\'os \cite{Ko} showed $p_n=O(n^{-1/2})$.  This bound was significantly improved by Kahn, Koml\'os, and Szemer\'edi in the 90's to $p_n \le 0.999^n.$
About ten years ago, Tao and Vu \cite{TV} improved the bound to $p_n\le (3/4+o(1))^n$. We also refer the reader to \cite{RV} by Rudelson and Vershynin for implicit bounds of type $e^{-cn}$.  The most recent record is due to and Bourgain, Vu and Wood \cite{BVW}, who show:
\begin{theorem}\label{theorem:BVW} 
  $$p_n \le \left(\frac{1}{\sqrt{2}}+o(1)\right)^n.$$  
\end{theorem}
These results imply that with very high probability the linear map $M_{n \times n}$ is injective over $\Z^n$ (and hence the lattice $M_{n\times n}(\Z^n)$ has full rank in $\Z^n$.) Another fundamental question of interest is the surjectivity onto $\Z^n$, more specifically:

{\it Is it true that with high probability $M_{n \times n}$ is also surjective over $\Z^n$ (in other words, the quotient group $\Z^n/M_{n \times n}(\Z^n)$  is trivial)?}

Unfortunately, the answer to this question turns out to be negative: with high probability $M$ is never surjective over $\Z^n$. To explain this at the heuristic level, assume that the vector $\Be_1=(1,0,\dots,0)$ is in the image space $M_{n \times n}(\Z^n)$, then (assuming that $M_{n\times n}$ is non-singular)
$$\Bx = M_{n \times n}^{-1}(\Be_1) = ((M_{n \times n}^{-1})_{11},\dots, (M_{n \times n}^{-1})_{1n})^T \in \Z^n.$$ 
However, we have $(M_{n \times n}^{-1})_{1i} = \frac{\det(M^{1i})}{\det(M_{n \times n})}$, where $M^{1i}$ is the matrix obtained from $M_{n \times n}$ by removing the first row and the $i$-th column. By the co-factor expansion
\begin{equation}\label{eqn:det}
  \det(M_{n \times n})= \sum_{i=1}^n (-1)^{i-1} M_{1i}\det(M^{1i}).
\end{equation} 
But as the $M_{1i}$ are independent from the submatrices $M^{1i} ,1\le i\le n$, it is highly unlikely that the random sum $|\sum_{i=1}^n (-1)^{i-1} M_{1i}\det(M^{1i})|$ becomes  smaller than all $|\det(M^{1i})|$ so that the components $\frac{\det(M^{1i})}{\det(M_{n \times n})}$ of $\Bx$ are all integral.

Having seen that $M_{n\times n}: \Z^n \to \Z^n$ is unlikely to be surjective, it is natural to think of rectangular matrices $M_{n \times (n+u)}: \Z^{n+u} \to \Z^n$ which might have better chance to be surjective. In fact, in the past several years there have been exciting developments (see for instance \cite{S-thesis,M2,W0,W1}) in the study of $M_{n \times (n+u)}(\Z^{n+u})$ for various ensembles of $M_{n \times (n+u)}$. For instance a special version of a recent result by Wood \cite[Corollary 3.4]{W1} shows:

\begin{theorem}\label{theorem:cok:B} Let $u\ge 0$ be a fixed integer. Let $M_{n \times (n+u)}$ be a random matrix with entries being iid copies of an $\alpha_n$-balanced random variable  of fixed $\alpha_n>0$. Let $P$ be a finite set of primes, then
  $$\lim_{n\to \infty} \P\Big(\cok(M_{n\times (n+u)})_P \simeq \{id\}\Big) =  \prod_{p\in P} \prod_{k=1}^\infty (1-p^{-k-u}),$$ 
  where $G_P = \prod_{p\in P} G_p$ is the product of $p$-Sylow subgroups of $G$ and the cokernel $\cok(M_{n\times m})$ is the quotient group $\Z^n/M_{n \times m}(\Z^m)$. 
\end{theorem}

We remark that $P$ is fixed and $n\to \infty$ in this result. 
However as $P$ increases, the probability on the right hand side of the limit becomes arbitrarily small.
  Hence it follows that 
\[
  \limsup_{n\to \infty}  
  \P\Big(\cok(M_{n \times n}) \simeq \{id\} \Big)
  \leq \inf_{P} 
  \lim_{n\to \infty}  
  \P\Big(\cok(M_{n \times n})_P \simeq \{id\} \Big)
  =0
\]
which officially answers our question above.

In the opposite direction, it has been conjectured by Koplewitz \cite{S1,S-thesis} that
\begin{conjecture}\label{conj:surj} Let the matrix entries be iid copies of an $\alpha_n$-balanced random variable  of fixed $\alpha_n>0$. Then for any fixed constant $\eps>0$,
  \begin{equation*}\label{conj:surj:weak}
    \lim_{n\to \infty}  \P\Big(\cok(M_{ n \times \lfloor (1+\eps) n \rfloor}) \simeq \{id\} \Big)  = 1.
  \end{equation*}
  Also, with $u\to \infty$ together with $n$
  \begin{equation*}\label{conj:surj:strong}
    \lim_{n\to \infty}  \P\Big(\cok(M_{ n \times (n+u)}) \simeq \{id\} \Big)  = 1.
  \end{equation*}
\end{conjecture}  
To support these conjectures, Koplewitz himself showed in  \cite[Theorem 1]{S1} (see also \cite[Theorem 30]{S-thesis}) that $\P\Big(M_{ n \times  \lfloor (2+\eps) n \rfloor} \simeq \{id\} \Big)  \ge 1 - e^{-c_\eps n}$. In the same paper he also confirmed Conjecture \ref{conj:surj} for random matrices of entries distributed according to the Haar measure over the profinite completion $\widehat{\Z}$ of $\Z$.

In this note we confirm the first conjecture. In fact we are able to extend the result to very sparse matrices. More specifically,  we can assume $\xi$ to take integer values as in \eqref{eqn:xi} with 
\begin{equation}\label{eqn:log}
  \alpha_n \ge \frac{C_0 \log n}{n}
\end{equation}
for a sufficiently large constant $C_0$.

\begin{theorem}[Main result]\label{theorem:sur} Let $\xi$ be as in \eqref{eqn:log}. Assume furthermore that $\xi$ is bounded with probability one. Then for every $A, \eps_0>0$, there exist $B= B(A,C_0,\eps_0)$ and an absolute constant $c$ such that 
  $$\P\Big(\cok(M_{n \times (n+\lfloor B (\frac{\log n}{\alpha_n}\log(\frac{\log n}{\alpha_n}) +\log n)   \rfloor)}) \simeq \{id\}\Big) \ge 1 -O(n^{-A}+e^{-c \alpha_n n}).$$ 
  In particular, if $\alpha_n$ is fixed then
  \begin{equation}\label{sur}
    \P\Big (\cok(M_{ n \times \lfloor n + \log^{1+o(1)}n \rfloor}) \simeq \{id\}\Big)  \ge  1 - O(n^{-\omega(1)});
  \end{equation}
  as well as if $\alpha_n \ge \frac{\log^{O(1)}n}{n}$ then
  \begin{equation}\label{sur:sparse}
    \P\Big (\cok(M_{ n \times \lfloor (1+o(1)) n\rfloor}) \simeq \{id\}\Big)  \ge  1 - O(n^{-\omega(1)}).
  \end{equation}
\end{theorem}

Note that a balanced assumption on $\xi$ is necessary as the results no longer hold for instance if we work with the Bernoulli $\pm 1$ ensemble; in this case the matrix cannot be surjective modulo 2 for even $n$.  Note also by considering the $\left\{ 0,1 \right\}$ ensemble with $\alpha_n = \P(\xi = 1),$ we can see that roughly the stated number of additional columns is necessary up to multiplicative constants, just by considering rows that are identically $0.$ 

We will also discuss an extension to a family of matrices of dependent entries, see Section \ref{section:dependent}. Our method is  short  and direct. We will first prove a slightly weaker version (Theorem \ref{theorem:simple}) by relying on a totally elementary lemma by Odlyzko (Lemma \ref{lemma:O}). We then refine the method by using a more involved result by Maples from \cite{M1} (Theorem \ref{theorem:corank'}). However, as \cite{M1} appears to be slightly incomplete and contains several (minor) errors, we will take this opportunity to recast Maples' proof toward our sparsest settings. Along the way, we show that this approach also yields a completely new singularity bound for sparse integral matrices. 

\begin{theorem}\label{theorem:M}  There exists an absolute constant $c>0$ such that as long as the entries of $M_{n\times n}$ are iid copies of $\xi$ distributed as in \eqref{eqn:log} (which is not necessarily bounded) then
  $$p_n \le e^{-c \alpha_n n}.$$
\end{theorem}
We notice that the recent paper \cite{BR} by Basak and Rudelson addressed the singularity (and in more general the least singular value) for  a general family of sparse matrices. Unfortunately, Theorem \ref{theorem:M} does not seem to follow from \cite{BR} because we have no restriction on the spectral norm of $M_{n\times n}$.

\section{Proof of  Theorem \ref{theorem:sur}} We assume that 
$$\P(|\xi| \le K_0) =1,$$
for some positive constant $K_0$. This assumption is only for  Theorem \ref{theorem:sur}, but not for Theorem \ref{theorem:M}.

A natural approach is to show that the equation system
$$M_{n \times m} \Bx = \Be_i,$$
has solutions $\Bx\in \Z^{m}$, for any standard unit vector $\Be_i=(0,\dots, 0,1,0,\dots, 0), 1\le i \le n$. However such an approach does not look simple as we would have to prove cancellation of extremely large numbers involving determinants of minors (see also the discussion around \eqref{eqn:det}). Instead, we will prove surjectivity by reducing our matrices over finite fields via the following result.

\begin{lemma}\cite[Lemma 5]{S1}\label{lemma:sur:equi} Let $m\ge n$. A matrix $M_{n \times m}: \Z^m \to \Z^n$ is surjective if and only if the modulo matrix $M_{n \times m}/p : \F_p^m \to \F_p^n$ is surjective for every prime $p$. Here $M_{n \times m}/p$ is the matrix over $\F_p$ given by $(M_{n \times m}/p)_{ij} = M_{ij} (\mod p)$.
\end{lemma}
\begin{proof}(of Lemma \ref{lemma:sur:equi}) Assume that  $M_{n \times m}/p : \F_p^m \to \F_p^n$ is surjective, then $M_{n \times m}/p$ contains a submatrix $M_{n\times n}'=M_{n \times n}'(p )$ (depending on $p$) of size $n\times n$ such that $\det(M_{n \times n}'/p)\neq 0 (\mod p)$. Thus $\det(M_{n \times n}') \neq 0$. This implies that the columns of $M_{n \times n}'$ generate $\Q^n$, and hence $M_{n \times n}'(\Z^n)$ is a full-rank integer lattice. In particular, the lattice co--volume $d=|\Z^n/M_{n \times m}(\Z^m)|$ (which is independent of $p$) is finite and $d$ divides $\det(M_{n \times n}')$ for all $p$.  
  Now assume that $d \neq 1$. Then as $\det(M_{n \times n}') \neq 0 (\mod p)$, $d$ is not divisible by $p$. But this holds for all prime $p$, a contradiction.
\end{proof}
By this lemma, for our problem we need to show that $M_{n \times m}/p : \F_p^m \to \F_p^n$ is surjective (or equivalently, $M_{n \times m}/p$ has rank $n$ in $\F_p^n$) for {\it every} prime $p$. This does not seem to be an easier task, but in what follows we show that there is a way to restrict the treatment to a set of a only a few primes.

Our first ingredient is the following simple bound (see also \cite[Lemma 3.9]{Ng-repulsion}).

\begin{lemma}[quadratic estimate]\label{lemma:lowrank} Let $p$ be a prime. Let $0<\eps_n<1$ be a given parameter that might depend on $n$. Let $M_{n \times n}$ be a matrix of size $n\times n$ whose entries are iid copies of a random variable $\xi$ from \eqref{eqn:log}. Then the probability that $M_{n \times n}$ has rank at most $(1-\eps_n)n$ in $\F_p^n$ is smaller than $e^{- \alpha_n \eps_n^2 n^2 +n}$.
\end{lemma}
To prove this result we rely on a useful result by Odlyzko \cite{KKS}.
\begin{lemma}\label{lemma:O} Let $H$ be a subspace of dimension $1\le d\le n$ in $\F_p^n$. Then if $X$ is a random vector whose entries are iid copies of a random variable $\xi$ from \eqref{eqn:log}, then
  $$\P(X\in H) \le (1-\alpha_n)^{n-d}.$$
\end{lemma}
We insert a proof of this well-known result here for completion.
\begin{proof}[Proof of Lemma \ref{lemma:O}] 
  Let $\left\{ H_1, H_2, \dots, H_d \right\}$ be a basis for $H.$
  By permuting coordinates, we may assume without loss of generality that the restrictions 
  $\tilde H_1, \tilde H_2, \dots, \tilde H_d$ of these vectors to the first $d$ coordinates are again linearly independent.
  Consider the event $X=(\xi_1,\dots,\xi_d, \xi_{d+1},\dots, \xi_n)^t \in H$. 
  From the linear independence of $\left\{ \tilde H_i \right\}_1^d,$ there are unique $c_1,c_2,\dots,c_d \in \F_p$ so that
  \[
    (\xi_1,\dots, \xi_d)^t = \sum_{i=1}^d c_i \tilde H_i.
  \]
  Hence conditioning on $(\xi_1,\dots, \xi_d)$, if $X=(\xi_1,\dots, \xi_d, \xi_{d+1}, \dots, \xi_n)^t \in H$ then
  \[
    X = \sum_{i=1}^d c_i H_i.
  \]
  In particular, each value $\xi_{d+1}, \xi_{d+2}, \dots, \xi_n$ is determined.
  However the probability of each of these events is at most $1-\alpha_n$, and so by independence the event $X \in H$  holds with probability at most $(1-\alpha_n)^{n-d}$. 
\end{proof}
Now we turn to the quadratic estimate. 
\begin{proof}[Proof of Lemma \ref{lemma:lowrank}] Let $d=\lfloor (1-\eps_n)n\rfloor$. Assume that the columns $X_{i_1},\dots, X_{i_d}$ span the column space of $M_{n \times n}$. For now assume that $\{i_1,\dots, i_d\} =\{1,\dots, d\}$. Let $H$ be the subspace spanned by $X_1,\dots, X_d$. We are considering the event $\CE_{1,\dots, d}$ that $X_i \in H, d+1 \le i \le n$. By Lemma \ref{lemma:O}, for any $i\ge d+1$, 
  $$\P(X_i \in H) \le (1-\alpha_n)^{n-d} \le e^{- \alpha_n \eps_n n}.$$
  Applying this bound for $d+1\le i\le n$ and using independence we obtain
  $$\P_{X_i, d+1\le i\le n}\Big(\CE_{1,\dots, d} | X_1,\dots, X_d\Big) = \P_{X_i, d+1\le i\le n} \Big(X_{d+1}, \dots, X_n \in H | H\Big) \le e^{- \alpha_n \eps_n^2 n^2}.$$
  Taking the union bound over at most $2^n$ choices of $\{i_1,\dots, i_d\}$ we conclude the proof.
\end{proof}

\subsection{A simpler result} To get the main idea, in this subsection we show:

\begin{theorem}\label{theorem:simple}
  For every $A>0$, there exists sufficiently large $B$ such that 
  $$\P\Big(\cok(M_{n \times (n+u)}) \simeq \{id\}\Big) \ge 1 -n^{-A},$$ 
  where 
  $$u=\left\lfloor B \frac{\log^2n}{\alpha_N} +  \sqrt{ \frac{n\log n}{\alpha_n}} \right\rfloor.$$
\end{theorem}
Note that $u$ does not drop below $\sqrt{n \log n}$. 

In what follows we prove Theorem \ref{theorem:simple}. The same argument will also be used to deal with matrices of dependent entries. Let $\CP_n$ be the set of primes up to  $(K_0n)^{n/2}$ 
\begin{equation}\label{eqn:P_n}
  \CP_n :=\Big\{ p \mbox{ prime },  p\le (K_0n)^{n/2}\Big\}.
\end{equation}
By taking the union bound, Lemma \ref{lemma:lowrank} then implies:
\begin{corollary}\label{cor:lowrank:simple} Let $\CE$ be the event that the matrix $M_{n\times n}$ has rank at least $(1-\eps_n)n$ in $\F_p^n$ for all $p\in \CP_n$. Then
  $$\P(\CE) \ge 1-e^{- \alpha_n \eps_n^2 n^2 +n \log n+n+ n \log K_0/2}.$$
\end{corollary}
Set
$$\eps_n := \sqrt{\frac{ 3\log n}{ \alpha_n n}}.$$
With this value of $\eps_n<1$, Corollary \ref{cor:lowrank:simple} implies
$$\P(\CE) \ge 1- e^{-n \log n}.$$

\begin{lemma}[surjectivity for small primes]\label{lemma:sur:smallp} For any $A$ there is a $B$ sufficiently large so that with probability at least $1-n^{-A}$, the random matrix $M_{n \times (n+u)}$ is surjective over $\F_p^n$ for all $p\in \CP_n$ simultaneously, and $u=u(B)$ is as in Theorem \ref{theorem:simple}.
\end{lemma}

\begin{proof}(of Lemma \ref{lemma:sur:smallp}) It suffices to show that with high probability $M_{n \times (n+u) }$ has full rank in $\F_p^n$ (which would then imply surjectivity in $\F_p^n$). 

  We consider the submatrix $M_{n\times n}$, the restriction of $M_{n \times  (n+u) }$ to the first $n$ columns. Let $\CE$ be the event defined in Corollary \ref{cor:lowrank:simple}, i.e.\ that $M_{n\times n}$ has rank at least $(1-\eps_n)n$ over $\F_p^n$ for all $p\in \CP_n$. We thus have
  $$\P_{M_{n\times n}}(\CE) \ge 1- e^{-n \log n}.$$
  Consider also the event $\CE_{\neq 0}$ that $\det(M_{n\times n}) \neq 0$, where by Theorem \ref{theorem:M}
  $$\P(\CE_{\neq 0}) \ge 1-n^{-A},$$
  provided that $\frac{C_0 \log n}{n} \le \alpha$ for large $C_0$.

  Now we condition on $M_{n \times n}$ satisfying $\CE$ and $\CE_{\neq 0}$ and show that with high probability (with respect to the last $u$ columns) that $M_{n \times (n+ u)}$ is surjective over all $\F_p^n, p\in \CP_n$. 

  Let $\CP^\ast = \CP^\ast(M_{n\times n})$ be the collection of prime divisors of $\det(M_{n\times n})$. Because $|\det(M_{n\times n})| \le (K_0n)^{n/2}$ by the Hadamard bound, the random set $\CP^\ast$ has small size, say
  $$|\CP^\ast| \le n^2.$$

  {\bf Case 1.} When $p \in \CP_n$ but $p \notin \CP^\ast$, then $M_{n \times n}$ has full rank in $\F_p^n$, and so does $M_{n \times (n +u)}$.

  {\bf Case 2.} Consider $p \in \CP^\ast$, we estimate the probability of the event $\CE_p$ that $M_{n \times (n+u)}/p$ has full rank. 

  Let $H_0 \subset \F_p^n$ be the column subspace of $M_{n\times n}$, for which by assumption 
  $$d_0:=n-\dim(H_0) \le n-(1-\eps_n) n \le \sqrt{\frac{3n \log n}{\alpha_n}}.$$

  We next expose the remaning $u$ vectors in groups. For $1\le i\le d_0$, at step $i$ we will add $k_i$ column vectors $X_{n+\sum_{l=1}^{i-1}k_l+j}, 1\le j \le k_i$ to the set of already exposed column vectors $X_1,\dots, X_{n+\sum_{l=1}^{i-1}k_l}$, where 
  $$k_i:= \left\lceil  \frac{B\log n}{ \alpha d_{i-1}} \right\rceil,$$
  and where $d_{i-1}$ is the codimension of the subspace $H_{i-1}$ generated by $\lang X_1,\dots, X_{n+\sum_{l=1}^{i-1}k_l} \rang$. Notice that in this exposing process the choice of $k_i$ depends on $d_{i-1}$, a decreasing sequence throughout the process.

  Next let $\CF_i$ be the event that $\dim(H_i) \ge \dim(H_{i-1})+1$. In other words, $\CF_i$ is the event that after adding the vectors of group $i$ we have a strict decrease in the co-rank,
  $$d_i \le d_{i-1}-1.$$

  Assuming that $\dim(H_{i-1})<n$, then by Lemma \ref{lemma:O}, and by independence of the column vectors,  
  \begin{align*}
    \P\Big(\CF_i|\wedge_{j=0}^{i-1}\CF_{j} \wedge \CE \wedge \CE_{\neq 0}, \dim(H_{i-1})<n\Big) &\ge 1 -((1-\alpha_n)^{\codim(H_{i-1})})^{k_i}  \\
    & \ge 1 -((1-\alpha_n)^{d_{i-1}})^{k_i}\\
    & \ge 1-n^{-B}.
  \end{align*}
  By Bayes' rule, with probability at least $(1-n^{-B})^{d_0} \ge 1-n^{-B+1}$, after  adding 
  $\sum_i k_i \le  \sum_i \frac{B \log n}{\alpha_n d_i}+1 \le  \frac{B \log n}{\alpha_n} \log d_0 + d_0$ columns, the matrix $M_{n \times (n+  \frac{B \log n}{\alpha_n} \log d_0 + d_0)}$ has full rank in $\F_p^n$.

  Taking union bound over all primes $p\in \CP^\ast$, we obtain that with probability at least $1-n^{-B+3}$ the obtained matrix has full rank in $\F_p^n$ for all $p \in \CP^\ast$.

  By {\bf Case 1.} and  {\bf Case 2.}, we have seen that with $M_{n \times n}$ satisfying $\CE$ and $\CE_{\neq 0}$, the matrix $M_{n \times (n+ u)}$ is surjective simultaneously over $\F_p^n$ for all $p \in \CP_n$ with the desired probability. The proof is then complete  after unfolding the conditioning on $M_{n\times n}$ (using Corollary \ref{cor:lowrank:simple}).
\end{proof}

\begin{proof}(of Theorem \ref{theorem:simple}) We condition on the event $\CE_{\neq 0}$. Note that with probability one $|\det(M_{n \times n})| \le (K_0n)^{n/2}$. This shows that with prime $p> (K_0n)^{n/2}$, $\det(M_{n \times n}) \neq 0 (\mod p)$. Hence on $\CE_{\neq 0}$ the matrix $M_{n\times n}$ is surjective over $\F_p^n$ for all $p \ge (K_0n)^{n/2}$.

  Furthermore, Lemma \ref{lemma:sur:smallp} implies that with probability at least $ 1-n^{-A}$, for all $p\in \CP_n$ the random matrix $M_{n \times (n+ u)}$ is surjective over $\F_p^n$. Hence altogether our matrix $M_{n \times (n+u)}$ is surjective in $\Z^n$ by Lemma \ref{lemma:sur:equi}.
\end{proof}


\subsection{Proof of Theorem \ref{theorem:sur}} Now we turn to our main theorem, where the proof is similar  but instead of Lemma \ref{lemma:O} we will be using the following result by Maples (see either \cite[Theorem 1.2]{M1} or \cite[Corollary 1.3]{M2}.) 

\begin{theorem}\label{theorem:corank'} Let $p$ be any prime. Assume that the entries of $M_{n\times n}$ are iid copies of $\xi$ from \eqref{eqn:xi} with $\alpha$ from \eqref{eqn:log}. Then for all $k\le \eta n$ with a sufficiently small absolute constant $\eta$ we have
  \begin{equation}\label{eqn:corank':bound}
    \P(\rank(M_{n \times n}/p) = n-k) = O\Big(n^k (p^{-k^2} + e^{-c \alpha n})\Big).
  \end{equation}
\end{theorem}

In fact \cite[Corollary 1.3]{M2} says much more, that the bound is precisely 
\begin{equation}\label{eqn:corank:bound}
  p^{-k^2}  \prod_{i=1}^k (1-p^{-i})^{-1}   \prod_{i=k+1}^\infty (1-p^{-i}) + O(e^{-c \alpha n}).
\end{equation}

However, we will not need this later result (given that it has not been formally verified, especially for the sparse case). Note that \eqref{eqn:corank:bound}, in its limit form ($n\to \infty$), is a simple consequence of the aforementioned paper \cite[Corollary 3.5]{W1} by Wood. Back to \cite{M1}, as this paper has some mistakes (for instance \cite[Proposition 2.3]{M1} is incorrect, see the appendix for further discussion), for transparency we will recast an almost complete proof of Theorem \ref{theorem:corank'} in the appendix. Theorem \ref{theorem:corank'} and Theorem \ref{theorem:M} will then follow as a byproduct.


Let $\CP_n'$ be the set of primes up to $e^{c \alpha_n n/2}$, where $c$ is the sufficiently small constant from Theorem \ref{theorem:corank'}, i.e.
$$\CP_n' :=\Big\{ p \mbox{ prime },  p\le e^{c \alpha_n n/2} \Big\}.$$ 
Note that $\CP_n' \subset \CP_n$ (defined in \eqref{eqn:P_n}). By applying \eqref{eqn:corank':bound} with $k = C_1\log n$ for a sufficiently large constant $C_1$ to each $p\in \CP_n'$ and taking the union bound 
$$\sum_{p \in \CP_n'}  n^k(p^{-k^2} + O(e^{-c  \alpha_n n}))  = n^{-\omega(1)}.$$

\begin{corollary}\label{cor:lowrank:main} Let $\CE'$ be the event that the matrix $M_{n \times n}$ has rank at least $n -C_1\log n$ in $\F_p^n$ for all $p\in \CP_n'$. Then
  $$\P(\CE') \ge 1- n^{-\omega(1)}.$$
\end{corollary}
We next prove an analog of Lemma \ref{lemma:sur:smallp}. Set 
$$u =  \left\lfloor B\cdot \biggl(\frac{\log n}{\alpha_n}\log  \frac{\log n}{\alpha_n}+\log n\biggr) \right\rfloor,$$
for a sufficiently large constant $B$.
\begin{lemma}\label{lemma:sur:smallp:log} With probability at least $1-n^{-A}$, the random matrix $M_{n \times (n + u)}$ is surjective over $\F_p^n$ for all $p\in \CP_n$ simultaneously.
\end{lemma}

\begin{proof}[Proof of Lemma \ref{lemma:sur:smallp:log}] Again, if suffices to show that with high probability$M_{n \times (n + u)}/p$ has full rank in each $\F_p^n$.

  We consider the submatrix $M_{n\times n}$, the restriction of $M_{n \times (n + u)}$ to the first $n$ columns. Let $\CE'$ be the event implied by Corollary \ref{cor:lowrank:main} that this matrix $M_{n\times n}$ has rank at least $n -
  C_1\log n$ over $\F_p^n$ for all $p\in \CP_n'$. We thus have
  $$\P_{M_{n\times n}}(\CE') \ge 1- n^{-\omega(1)}.$$
  Consider also the event $\CE_{\neq 0}$ that $\det(M_{n\times n}) \neq 0$ from Theorem \ref{theorem:M}. Conditioning on $M_{n \times n}$ satisfying $\CE'$ and $\CE_{\neq 0}$, we will show that with high probability (with respect to the last $u$ columns) that $M_{n \times (n+ u)}$ is surjective over all $\F_p^n, p\in \CP_n$. 

  To do this, similarly to the proof of Lemma \ref{lemma:sur:smallp}, let $\CP^\ast = \CP^\ast(M_{n\times n})$ be the collection of prime divisors of $\det(M_{n\times n})$, then clearly the random set $\CP^\ast$ has size at most $n^2$. 

  {\bf Case 1.} When $p \in \CP_n$ but $p \notin \CP^\ast$, then $M_{n \times n}$ has full rank in $\F_p^n$, and so does $M_{n \times (n +u)}$.

  {\bf Case 2.} Consider $p \in \CP^\ast$, we estimate the probability of the event $\CE_p$ that $M_{n \times (n + u)}$ has full rank over $\F_p^n$. For this, first note that under $\CE'$, if $p \in \CP_n'$ (that is $p \le e^{c \alpha_n n/2}$) then the corank of $M_{n \times n}$ over $\F_p^n$ is at most $C_1\log n$. Now if $e^{c \alpha_n n/2} <p\le (K_0n)^{n/2}$, as $|\det(M_{n \times n})| \le (K_0n)^{n/2}$, the corank of $M_{n\times n}$ over $\F_p^n$ for these large $p$ must be at most $\frac{\log (K_0n)^{n/2}}{\log p} \le (c \alpha_n) ^{-1} \log (K_0n)$. So in either case the corank is at most $(c\alpha_n)^{-1} \log n +C_1 \log n$.

  Let $H_0 \subset \F_p^n$ be the column subspace of $M_{n\times n}$, for which by assumption 
  $$d_0:=n-\dim(H_0) \le  (c\alpha_n)^{-1} \log n +C_1\log n.$$

  Similarly to the proof of Theorem \ref{theorem:simple},  for $1\le i\le d_0$, we will add $k_i= \lceil  \frac{B\log n}{ \alpha_n d_{i-1}} \rceil$ column vectors $X_{n+\sum_{l=1}^{i-1}k_l+j}, 1\le j \le k_i$ to the set of already exposed column vectors $X_1,\dots, X_{n+\sum_{l=1}^{i-1}k_l}$, where $d_{i-1}$ is the codimension of the subspace $H_{i-1}$ generated by $\lang X_1,\dots, X_{n+\sum_{l=1}^{i-1}k_l} \rang$. 

  Let $\CF_i$ be the event that $\dim(H_i) \ge \dim(H_{i-1})+1$. By Lemma \ref{lemma:O}, and by independence of the column vectors,  
  $$\P\Big(\CF_i|\wedge_{j=0}^{i-1}\CF_{j} \wedge \CE' \wedge \CE_{\neq 0}, \dim(H_{i-1})<n\Big) \ge 1 -((1-\alpha)^{\codim(H_{i-1})})^{k_i}  \ge 1-n^{-B}.$$
  By Bayes' rule, with probability at least $(1-n^{-B})^{d_0} \ge 1-n^{-B+1}$, after  adding $\sum_i k_i = O(\frac{\log n}{\alpha_n}\log d_0 + d_0):=u$ columns, the matrix $M_{n \times (n+ u)}$ has full rank in $\F_p^n$. (It is possible to improve the total number of extra columns by a more careful analysis of the $d_i$ but we will not do so here for simplicity.)

  Taking the union bound over all primes $p\in \CP^\ast$, we obtain that with probability at least $1-n^{-B+3}$ the matrix $M_{n \times (n+u)}$ has full rank in $\F_p^n$ for all $p \in \CP^\ast$.

  We have seen that with $M_{n \times n}$ satisfying $\CE$ and $\CE_{\neq 0}$, the matrix $M_{n \times (n+ u)}$ is surjective simultaneously over $\F_p^n$ for all $p \in \CP_n$ with the desired probability. The proof is then complete  after unfolding the conditioning on $M_{n\times n}$, knowing that these events hold with very high probability.
\end{proof}

Finally, for Theorem \ref{theorem:sur}, conditioning on the event $\CE_{\neq 0}$, with prime $p> (K_0n)^{n/2}$ we have $\det(M_{n \times n}) \neq 0 (\mod p)$, and  hence on $\CE_{\neq 0}$ the matrix $M_{n\times n}$ is surjective over $\F_p^n$ for all $p \ge (K_0n)^{n/2}$.

On the other hand, Lemma \ref{lemma:sur:smallp:log} implies that with probability at least $ 1-n^{-A}$, for all $p\in \CP_n$ the random matrix $M_{n \times (n+ u)}$ is surjective over $\F_p^n$.

\section{Some remarks}\label{section:dependent}
We have studied random matrices of independent entries. It is natural to consider Conjecture \ref{conj:surj} for other families of matrices of dependent entries. Here we discuss one such model.

Let $M_{n\times n}$ be a random symmetric matrix, where for simplicity we assume that the entries $(M_{n\times n})_{ij}, 1\le i \le j \le n$ are iid copies of a bounded random variable $\xi$ from \eqref{eqn:xi} with fixed $\alpha$. It follows from \cite{Ng-sym,Ver} that for this model the singularity probability can be bounded by 
\begin{equation}\label{eqn:NgV} 
  p_n =n^{-\omega(1)}.
\end{equation}

Heuristically, arguing similarly to \eqref{eqn:det} (where we expose both columns and rows at the same time to obtain a quadratic variant of \eqref{eqn:det}), we can show that with high probability the matrix $M_{n\times n}$ is not surjective over $\Z^n$. Actually an analog of Theorem \ref{theorem:cok:B} has been established in \cite{M2} for this model \footnote{To be more precise, M. M. Wood studied the Laplacian, but her result also covers the non-normalized ensemble.}, which confirms the above heuristic. However, we will show that by adding a couple of few more (say) independent rows, the matrix becomes surjective. 

\begin{theorem}\label{theorem:sym} Let $M_{n \times (n+u)}$ be a random matrix where its restriction $M_{n \times n}$ to the first $n$ columns is a symmetric matrix as above, and the last $u$ columns are independent with entries being iid copies of $\xi$. Then for any $A>0$, there exists $B$ such that for $u = \lfloor B \sqrt{ n \log n} \rfloor$ 
  $$\P\Big (\cok(M_{ n \times (n +u)}) \simeq \{id\}\Big)  \ge  1 - n^{-A}.$$
\end{theorem}

To justify this result,  we establish the following analog of Lemma \ref{lemma:lowrank}. 

\begin{lemma}[quadratic estimate]\label{lemma:lowrank:sym} Let $p$ be a prime. Let $0<\eps_n<1$ be a given parameter that might depend on $n$. Let $M_{n \times n}$ be a symmetric matrix  where $(M_{n\times n})_{ij}, 1\le i \le j \le n$ are iid copies of a bounded random variable $\xi$ from \eqref{eqn:xi} with fixed $\alpha_n$. Then the probability that $M_{n \times n}$ has rank at most $(1-\eps_n)n$ in $\F_p^n$ is smaller than $e^{- \alpha _n\eps_n^2 n^2/2 +n}$.
\end{lemma}

\begin{proof}[Proof of Lemma \ref{lemma:lowrank:sym}] Let $d=\lfloor (1-\eps_n)n\rfloor$. Assume that the columns $X_{i_1},\dots, X_{i_d}$ spans the column space of $M_{n \times n}$. For now assume that $\{i_1,\dots, i_d\} =\{1,\dots, d\}$. Let $H$ be the span of $X_1,\dots, X_d$. We are considering the event $\CE_{1,\dots, d}$ that $X_i \in H, d+1 \le i \le n$. Now as $X_i$ is dependent on $H$, we cannot estimate the probability of $X_i \in H$ directly by Odlyzko's bound. However, we can get rid of the dependence by deleting the corresponding common entries as below. 

  For $1\le j\le n-d$ set 
  $$I_{d+j}:=\{1,\dots,d+j\}.$$ 
  For any $X\in \F_p^n$ and $J \subset [n]$ we denote $X|_J$ by the restriction of $X$ over the components indexed by $J$. For convenience we also denote $H|_J$ by the subspace generated by $X_1|_J,\dots, X_d|_J$. Assume that $X_{d+1},\dots, X_n \in H$, then the following holds
  \begin{itemize}
    \item $X_{d+2}|_{I_{d+1}} \in H_{I_{d+1}}$, and more generally $X_{d+j+1}|_{I_{d+j}} \in H|_{I_{d+j}}, 1\le j\le n-d-1$;
      \vskip .1in
    \item the vector $X_{d+j+1}|_{I_{d+j}}$ is independent of $H|_{I_{d+j}}$;
      \vskip .1in
    \item the vectors $X_{d+j+1}|_{I_{d+j}}, 1\le j\le n-d-1$ are mutually independent.
  \end{itemize} 
  Now as $H|_{I_{d+j}}$ has rank at most $d$ in $\F_p^{d+j}$, by Lemma \ref{lemma:O} we have
  $$\P(X_{d+j+1}|_{I_{d+j}} \in H|_{I_{d+j}}) \le (1-\alpha_n)^{j}.$$
  Applying this bound for $1\le j\le n-d-1$ and using the independence of $X_{d+j+1}|_{I_{d+j}}$, we obtain
  $$\P_{X_i, d+1\le i\le n}\Big(\CE_{1,\dots, d} | X_1,\dots, X_d\Big) \le \prod_{j=1}^{n-d-1} (1-\alpha_n)^j \le e^{- \alpha_n \eps_n^2 n^2/2}.$$
  Taking union bound over at most $2^n$ choices of $\{i_1,\dots, i_d\}$ we conclude the proof.
\end{proof}

We can now complete the proof of Theorem \ref{theorem:sym} verbatim as in the proof of Theorem \ref{theorem:simple} with fixed $\alpha_n L$. Indeed, Corollary \ref{cor:lowrank:simple} follows from Lemma \ref{lemma:lowrank:sym}, and Lemma \ref{lemma:sur:smallp} can be justified similarly (conditioning on \eqref{eqn:NgV}) because the last $u$ columns are mutually independent, and are independent from $M_{n\times n}$.

\vskip .15in

{\bf Acknowledgement.} The authors thank Kyle Luh for helpful comments.

\appendix

\section{The corank estimate: proof  proof of Theorem \ref{theorem:corank'}}\label{appendix:intro}

We will work in a more general setting. Let $q=p^f$ be a prime power and $\F_q$ be the finite field with $q$ elements. We say that a probability distribution $\mu$ in $\F_q$ is {\it $\alpha_n$-balanced} (for some $0< \alpha_n<1$) if for every additive subgroup $T$ in $\F_q$ and $s\in \F_q$
$$\mu(s+T) \le 1-\alpha_n.$$
In the general finite field setting, we will assume
\begin{equation}\label{eqn:alpha:q}
  \alpha_n \ge n^{-1/2+\eps} \mbox{ for any $\eps>0$ }.
\end{equation}

In the more specific setting when $q=p$ (which is the setting of Theorem \ref{theorem:corank'}), as there is no non-trivial additive subgroup in $\F_p$, we will assume 
$$\max_{x\in \F_p} \mu(x) = 1 -\alpha_n$$ 
where 
\begin{equation}\label{eqn:alpha'}
  \alpha_n \ge \frac{C_0 \log n}{n}, \mbox{ for a sufficiently large constant $C_0$.}
\end{equation}

In what follows $M_{n \times n}$ is a random matrix where the entries are independent and identically distributed according to an $\alpha_n$-balanced $\mu$ either from \eqref{eqn:alpha:q} or \eqref{eqn:alpha'}, and $n \to \infty$. Notice that in either case, we do not assume the support of $\mu$ to be bounded. Recall that $X_1,\dots, X_{n}$ are the columns of $M_{n\times n}$ and $W_{n-k}$ is the subspace $\lang X_1,\dots, X_{n-k} \rang$ generated by $X_1,\dots, X_{n-k}$. Our first goal is to reprove the following variant of \cite[Proposition 2.1]{M1} and \cite[Proposition 2.1.1]{M-thesis}. 

\begin{theorem}\label{theorem:corank:main} Assume that $\mu$ is distributed according to either \eqref{eqn:alpha:q} or \eqref{eqn:alpha'} depending on $q$. Then there exist positive constants $c,\eta$ such that the following holds for $1\le k\le \eta n$: there exists an event $\CE_{n-k}$ on the $\sigma$-algebra generated by $X_1,\dots, X_{n-k}$ of probability at least $1- e^{-c\alpha_n n}$ such that for any $k\le k_0\le \eta n$
  $$\P_{X_{n-k+1}}\Big(X_{n-k+1} \in W_{n-k} \big| \CE_{n-k} \wedge \codim(W_{n-k})=k_0\Big) = q^{-k_0} + O(e^{-c \alpha_n n}).$$
\end{theorem}
Notice that there are some modifications of this result compared to the original statement by Maples in \cite[Proposition 2.1]{M1} or \cite[Proposition 2.1.1]{M-thesis} that
\begin{enumerate}[(i)]
  \item the statement also holds when the codimension of $W_{n-k}$ is not necessarily $k$;
    \vskip .1in 
  \item the statement also holds for sparse settings such as \eqref{eqn:alpha:q} and \eqref{eqn:alpha'}.
\end{enumerate}

Note that (i) is not new as it also appeared in a subsequent (unpublished) preprint by Maples (\cite[Proposition 3.1]{M2}). We then deduce Theorem \ref{theorem:corank'} restated here for finite field.

\begin{cor}\label{cor:corank'}  Assume that $\mu$ is distributed according to either \eqref{eqn:alpha:q} or \eqref{eqn:alpha'}. Assume that $k\le \eta n$ for some sufficiently small $\eta$, then 
  \begin{equation}\label{eqn:corank'':bound}
    \P\Big(\rank(M_{n \times n}) = n-k\Big) = O\Big(n^k (q^{-k^2} + e^{-c \alpha n})\Big).
  \end{equation}
\end{cor}
It seems plausible to get rid of the factor $n^k$ here (especially for fixed $\alpha$) but we do not attempt to do so, as the reader can check that any improvement along this line has little affect on the bounds in Theorem \ref{theorem:sur}.

\begin{proof}(of Corollary \ref{cor:corank'})  The event $\rank(M_{n \times n}) = n-k$ implies that there exist $k$ column vectors $X_{i_1},\dots, X_{i_k}$ which belong to the subspace of dimension $n-k$ generated by the remaining column vectors $X_i, i \neq i_1,\dots, i_k$. With a loss of a factor of $n^k$ in probability, we can assume that $\{i_1,\dots, i_k\}=\{n-k+1,\dots, n\}$. We then use Theorem \ref{theorem:corank:main} to show
  \begin{align*}
    &\P\Big(X_{n-k+1},\dots, X_n \in W_{n-k} \wedge \codim(W_{n-k})=k\Big) \\ 
   =&\P\Big(X_{n-k+1},\dots, X_n \in W_{n-k} \wedge \CE_{n-k} \wedge  \codim(W_{n-k})=k\Big) 
    +O(e^{-c \alpha_n n}) \\
    \le &\P\Big(X_{n-k+1},\dots, X_n \in W_{n-k}|\CE_{n-k} \wedge  \codim(W_{n-k})=k\Big) + O(e^{-c \alpha_n n})\\
    \le &\Big(q^{-k} + O(e^{-c \alpha_n n})\Big)^k + O(e^{-c \alpha_n n})=O(q^{-k^2} + e^{-c \alpha_n n}).
  \end{align*}
\end{proof}

Taking $k=1$ and $q=p \to \infty$ in Corollary \ref{cor:corank'} we then obtain Theorem \ref{theorem:M}.

\begin{cor}\label{cor:sing} Assume that the entries of $M_{n\times n}$ are iid copies of a discrete random variable $\xi$ taking integer values such that  
  $$\max_{x\in \Z} \P(\xi =x) \le 1 - \frac{C_0 \log n}{n}, \mbox{ for a sufficiently large constant $C_0$.}$$
  Then the matrix $M_{n \times n}$ is non-singular with probability at least $1 -e^{-c \alpha_n n}$.
\end{cor}
\begin{proof}(of Corollary \ref{cor:sing}) Choose a prime $p$ to be large such that $p \ge \max\{n^{2n}, |\det(M_{n\times n})|\}$ and $\supp(\xi) \subset (-p,p)$. It then suffices to show that $M_{n \times n}/p$ has full rank with probability at least $1 -e^{-c \alpha_n n}$. To this end, and by Corollary \ref{cor:fullrank} (to be discussed in the sequel), it suffices to bound the probability that $M_{n \times n}/p$ has corank $k$ between $1$ and $\eta n$, but then the statement follows from corollary \ref{cor:corank'} (stated for $\mu$ distributed according to \eqref{eqn:alpha'}) by taking union bound. 
\end{proof}

Note that the trick to pass the singularity problem over $\Z$ to over $\F_p$, and let $p \to \infty$, is not new. See for instance \cite{BVW, TV} and the references therein. 

Finally, we will also show the following more general variant of Theorem \ref{theorem:corank:main} for rectangular matrices $M_{n \times (n+u)}$ where $W_k=\lang X_{k+1},\dots, X_{n+u} \rang$. 
\begin{theorem}\label{theorem:corank:main'} Assume that $\mu$ is distributed according to either \eqref{eqn:alpha:q} or \eqref{eqn:alpha'} depending on $q$. Then there exist positive constants $c,\eta$ such that the following holds. Let $0\le u\le \eta n$ be given. Then for $k\le \eta n$, there exist an event $\CE_{n+u-k}$ on the $\sigma$-algebra generated by $X_1,\dots, X_{n+u-k}$ of probability at least $1- e^{-c\alpha_n n}$ such that for any $(k-u)^+\le k_0\le \eta n$
  $$\P_{X_{n+u-k+1}}\Big(X_{n+u-k+1} \in W_{n+u-k} \big| \CE_{n+u-k} \wedge \codim(W_{n+u-k})=k_0\Big) = q^{-k_0} + O(e^{-c \alpha_n n}).$$
\end{theorem}

The rest of the appendix is mainly dedicated to verify Theorem \ref{theorem:corank:main}. The proof of Theorem \ref{theorem:corank:main'} will be deduced shortly. As already mentioned, our approach mainly follows \cite{M1}. 
\vskip .1in

{\bf \underline{Part I: proof of Theorem \ref{theorem:corank:main}}}\label{appendix:proof}: In what follows the constants $\eta, \beta, \delta, d$ are sufficiently small but fixed (see for instance \eqref{eqn:d} for a choice of $d$), and $\alpha_n$ is allowed to depend on $n$ as from \eqref{eqn:alpha:q} or \eqref{eqn:alpha'} for sufficiently large constant $C_0$. The only place we have to treat \eqref{eqn:alpha:q} and \eqref{eqn:alpha'} separately is in the proof of Proposition \ref{prop:sparse} below.

We first note that Odlyzko's lemma in fact holds in any finite field.
\begin{lemma}\label{lemma:O'} For a deterministic subspace $V$ of $\F_q^n$ and a random vector $X$ of iid entries from an $\alpha_n$-balanced distribution
  $$\P(X \in V) \le (1-\alpha_n)^{\codim(V)}.$$
\end{lemma}

\begin{corollary}\label{cor:fullrank} Let $X_1,\dots, X_{n-k}$ be the columns of $M_{n\times n}$. Then the probability that $X_1,\dots, X_{n-k}$
  are linearly independent in $\F_q^n$ is at least $1 -n(1-\a_n)^{k}.$ 
\end{corollary}

\begin{proof}(of Corollary \ref{cor:fullrank}) Let $0\le i\le n-k-1$ be smallest such that $X_{i+1} \in \Sp(X_1,\dots,X_i)$. By Lemma \ref{lemma:O}, this event is bounded by $(1-\alpha_n)^{n-i}$. Summing over $0\le i\le n-k-1$, the probability under consideration is bounded by $n(1-\alpha_n)^{k}$. 
\end{proof}

\subsection{Sparse subspace} Let $0<\delta, \eta$ be small constants (independently from $\alpha_n$). Given a vector space $H \subset \F_q^n$, we say that $H$ is {\it $\delta$-sparse} if there is a non-zero vector $w$ with $|\supp(w)| \le \delta n$ (i.e. $w$ is {\it $\delta$-sparse}) such that $w \perp H$, where $\supp(w)$ is the set of non-zero coordinates of $w$.

\begin{prop}[random subspaces are not sparse]\label{prop:sparse} Let $0<\eps_0<1/2$ be any fixed constants. Then for any $0 \le \delta, \eta$ such that $ \delta+\eta \le \eps_0$, and with $\alpha_n$ from \eqref{eqn:alpha:q} or \eqref{eqn:alpha'} the following holds for $0\le k < \eta n$: with probability at least $1 -e^{-c\alpha_n n}$ with respect to $X_1,\dots,X_{n-k}$  the random subspace $W_{n-k}$ is not $\delta$-sparse. Here $c=c(\eps_0)$ is an absolute constant.
\end{prop}

\begin{proof}(of Proposition \ref{prop:sparse}) 
  For $\sigma \subset [n]$ with $1\le t=|\sigma| \le \delta n$, let $\CE_\sigma$ be the event that  $W_{n-k}$ is orthogonal to a vector $w$ with $\supp(w) =\sigma$, but is not orthogonal to any vector $w'$ with $|\supp(w')| \le t-1$. Note that in this case the $\sigma$-restricted vector $w|_\sigma$ is orthogonal to the $\sigma$-restricted column vectors $X_{1}|_\sigma,\dots, X_{n-k}|_\sigma$.

  The dimension of the annihilator of $W_{n-k}$ in $\F_q^\sigma$ 
    and the dimension of $\Sp\left( X_{1}|_\sigma,\dots, X_{n-k}|_\sigma \right)$ sum to $t.$  If the annihilator were more than $1$ dimensional, there would necessarily exist a nonzero linear combination of annihilators with support strictly contained in $\sigma.$ Hence it follows that the column vectors $X_{1}|_\sigma,\dots, X_{n-k}|_\sigma$ span a subspace of dimension $t-1$, and there are $t-1$ linearly independent column vectors $X_{i_1}|_\sigma,\dots, X_{i_{t-1}}|_\sigma$ in $\F_q^{\sigma}$.

      We therefore define the event $\CE_{\sigma, i_1,\dots, i_{t-1}}$
      to be that
      \begin{enumerate}
	\item $X_{i_1}|_\sigma,\dots, X_{i_{t-1}}|_\sigma$ are linearly independent,
	\item $X_i|_\sigma \in \Sp(X_{i_1}|_\sigma,\dots, X_{i_{t-1}}|_\sigma)$ for all $1 \leq i \leq n-k,$
	\item the annihilator of $ \Sp(X_{i_1}|_\sigma,\dots, X_{i_{t-1}}|_\sigma)$ in $\F_q^\sigma$ contains no nonzero vectors with support begin a proper set of $\sigma,$
      \end{enumerate}
      and observe that $\CE_\sigma$ is a union over all such $\CE_{\sigma, i_1,\dots, i_{t-1}}$.

  {\bf Case 1.} If $ 144 \alpha_n^{-1} \le t \le \delta n$, then by Theorem \ref{theorem:LO} (whose proof is given in {\bf Part II}),

  \[
    \P(\CE_{\sigma, i_1,\dots, i_{t-1}}) \le 
    \left(\frac{1}{q} + \frac{2}{\sqrt{\alpha_n t}}\right)^{n-k-t+1}\le 
    \left(\frac{1}{q} + \frac{2}{\sqrt{\alpha_n t}}\right)^{(1-\eps_0) n} \le  
    \left(\frac{2}{3}\right)^{n/2}  .
  \]
  Thus 
  \[
    \P(\CE_\sigma) \le 
    \sum_{i_1,\dots, i_{t-1}} \P(\CE_{\sigma, i_1,\dots, i_{t-1}})  \le 
    \binom{n-k}{t-1} \left(\frac{2}{3}\right)^{n/2}.
  \]
  So
  \[
    \sum_{\sigma, |\sigma| \ge 144 \alpha_n^{-1} }\P(\CE_\sigma) \le 
    \sum_{144 \alpha_n^{-1} \le 
  t \le \delta n}\binom{n}{t}\binom{n-k}{t-1} \left(\frac{2}{3}\right)^{n/2} \le \left(\frac{2}{3}\right)^{n/4}.
  \]
  provided that $\ep_0$ (and hence $\delta$) are sufficiently small.

  {\bf Case 2.} If $1\le t \le 144 \alpha_n^{-1}$.  We use the simple bound 
  $$\P\Big(X_i|_\sigma \in \Sp(X_{i_1}|_\sigma,\dots, X_{i_{t-1}}|_\sigma)\Big) \le 1-\alpha_n.$$ 
  Hence  
  $$\P(\CE_\sigma) \le \sum_{i_1,\dots, i_{t-1}} \P(\CE_{\sigma, i_1,\dots, i_{t-1}})  \le \binom{n-k}{t-1}  (1-\alpha_n)^{(1-\delta -\eta)n} \le \binom{n-k}{t-1}  (1-\alpha_n)^{(1-\eps_0) n} .$$
  Consequently,
  $$\sum_{\sigma, |\sigma| \le 144 \alpha_n^{-1} }\P(\CE_\sigma) \le \sum_{t\le  144 \alpha_n^{-1}} \binom{n}{t} \binom{n-k}{t-1}  (1-\alpha_n)^{(1-\ep_0) n}.$$

  {\bf Subcase 2.1.} Assume that $\alpha_n$ is from \eqref{eqn:alpha:q}. Then as $\alpha_n \ge n^{-1/2+\eps}$, the above can be easily bounded by
  $$\sum_{\sigma, |\sigma| \le 144 \alpha_n^{-1} }\P(\CE_\sigma) \le (1-\alpha_n)^{\eps_0 n/2}.$$

  {\bf Subcase 2.2.} Assume that $\alpha_n$ is from \eqref{eqn:alpha'}, 
  $$\frac{C_0 \log n}{n} \le \alpha_n = 1 -\max_{x} \mu(x) \le n^{-1/2+\eps}.$$
  We will rely on the following observation, which is a simple variant of Lemma \cite[Lemma 3.2]{BR}. 
  \begin{claim}\label{claim:zeros} The following holds with probability at least $1- e^{-c\alpha_n n}$ with respect to $X_{1},\dots, X_{n-k}$. For any $1\le t \le 144 \alpha_n^{-1}$, and any $\sigma \in \binom{[n]}{t}$, there are at least two columns $X_i,X_{i+1}$ whose restriction $(X_{i+1}-X_i)|\sigma$ has exactly one non-zero entry. 
  \end{claim}

  Suppose that $w \in \F_q^n$ has support $\sigma$ of size $t=|\sigma|=|\supp(w)|\le 144 \alpha_n^{-1}.$  Then,
  conditioning on the event in the lemma,
  there is some $1 \leq i \leq n-k-1$ so that $(X_{i+1}-X_i)|_\sigma$ has exactly one nonzero entry, and hence $w|\sigma$ is not orthogonal to it.
Hence, it cannot be simultaneously orthogonal to all $X_i, 1 \le i\le n-k$.  Thus, it suffices to prove the claim.

\begin{proof}[Proof of Claim \ref{claim:zeros}] For each $i\in \{1, 3,\dots, 2\lfloor (n-k-1)/2 \rfloor +1 \}$, consider the vectors $Y_i = X_{i+1}-X_i$. The entries of this vector are iid copies of the symmetrized random variable $\psi =\xi-\xi'$, where $\xi',\xi$ are independent and have distribution $\mu$. With $1-\alpha_n'=\P(\psi=0)$, then we have 
    $$\alpha_n \le \alpha_n'  \le 2\alpha_n$$
    as this can be seen by
    $$(1-\alpha_n)^2 \le \max_{x} \P(\xi= x)^2 \le \sum_x \P(\xi=x)^2=\P(\psi =0) \le \max_{x} \P(\xi= x) = 1- \alpha_n.$$
    Now let $p_\sigma$ be the probability that all $Y_i|_\sigma, i\in \{1, 3,\dots, 2\lfloor (n-k-1)/2 \rfloor +1 \}$ fail to have exactly one non-zero entry, then by independence of the columns and of the entries
    $$p_\sigma = (1- t \alpha_n' (1-\alpha_n')^{t-1})^{(n-k)/2} \le (1 - t\alpha_n' e^{- t\alpha_n'})^{n-k} \le e^{-n t\alpha_n' e^{-t \alpha_n'}/2}.$$
    Notice that as $1\le t \le 8^{3/\eps_0} \alpha_n^{-1}$, $e^{-t \alpha_n'}/2 \ge c$ for some positive constant $c$, and hence 
    $$ e^{-n t\alpha_n' e^{-t \alpha_n'}/2} \le (e^{ -c n \alpha_n' })^t \le n^{-cC_0 t /2} e^{ -c n \alpha_n/2 }.$$
    Thus 
    $$\sum_{1\le t \le 144 \alpha_n^{-1}}\sum_{\sigma \in \binom{[n]}{t}} p_\sigma \le \sum_{1\le t \le 144 \alpha_n^{-1}} \binom{n}{t}  e^{-(n-k) t\alpha_n e^{-t \alpha_n}} \le  \sum_{1\le t \le 144 \alpha_n^{-1}}  (n^t n^{-cC_0 t /2}) e^{ -c n \alpha_n/2} <e^{ -c n \alpha_n/2},$$
    provided that $C_0$ is sufficiently large.
  \end{proof}
  Our proof of Proposition \ref{prop:sparse} is then complete by combining the cases considered above.
\end{proof}

\begin{remark}
  Our treatment here is quite different from \cite[Section 2.1]{M1} as there is no need to use \cite[Proposition 2.3]{M1} (which states that ``Let $Z_1,\dots,Z_r$ be non-trivial iid random vectors in $\F_q^n$, then $\P(Z_1,\dots,Z_r\in V|Z_1,\dots, Z_r \mbox{ are linearly independent}) \le \P(Z\in V)^r$." An elementary counterexample to this proposition is that  $Z_i$ are chosen uniformly at random from $\{a_1,2a_1, 3a_1, a_2,a_3\}$, where $a_1 \notin V$ and $a_2,a_3$ are linearly independent in $V$: in this case the LHS probability bound is larger than the RHS \footnote{We thank M. M. Wood for pointing out the mistake, as well as for supplying a counterexample.}. Also, Maples did not provide {\it any treatment} for the sparse case (Case 2) for Proposition \ref{prop:sparse} in \cite{M1} or \cite{M-thesis}, which we have added. See also the remark after Lemma \ref{lemma:inverse:semi} below.
\end{remark}

To conclude, given constants $\eta, \delta$ and the parameter $\alpha_n$ from \eqref{eqn:alpha:q} or \eqref{eqn:alpha'}, let $\CE_{n-k, dense}$ denote the event considered in Proposition \ref{prop:sparse}, 
\begin{equation}\label{eqn:dense}
  \P(\CE_{n-k, dense}) \ge 1 - e^{-c \alpha_n n}.
\end{equation}

We next turn to another type of subspace.

\subsection{Semi-saturated subspace} Given $0<\alpha_n, \delta, d<1$. We call a subspace $V$ of co-dimension $k_0$ {\it semi-saturated} (or {\it semi-sat} for short), where $k_0\le \eta n$, if $V$ is not $\delta$-sparse and 
\begin{equation}\label{eqn:semi-sat:def}
  e^{-d \alpha_n n} <  |\P(X \in V) - \frac{1}{q^{k_0}}| \le  \frac{16}{q^{k_0}}.
\end{equation}
Here we assume 
$$e^{-d \alpha_n n}< \frac{16}{q^{k_0}}.$$
If this condition is not satisfied (such as when $q$ is sufficiently large), then the semi-saturated case can be omitted.

\begin{lemma}\cite[Proposition 2.5]{M1}\label{lemma:inverse:semi} For all $\beta>0$ and $\delta>0$ there exists $0<d=d(\beta, \delta)<1$ in the definition of semi-saturation and a deterministic set $\CR \subset \F_q^n$ of non $\delta$-sparse vectors and of size $|\CR| \le (2\beta^\delta)^n q^n$ such that  every semi-saturated $V$ is orthogonal to a vector $R \in \CR$. In fact the conclusion holds for any subspace $V$ satisfying the LHS of \eqref{eqn:semi-sat:def}.
\end{lemma}

In short, semi-saturated subspaces are necessarily orthogonal to one of a small number of non-sparse vectors in $\F_q^n$.  A proof of this result will be given in {\bf Part II} where we emphasize that $\a_n$ can be as small as \eqref{eqn:alpha'}, in contrast to the proof of \cite[Proposition 2.5]{M1} where $\a_n$ was treated as a constant. 

Let $\CF_{n-k,k_0, semi-sat}$ be the event that $\codim(W_{n-k})=k_0$ and $W_{n-k}$ is semi-saturated.

\begin{prop}\label{prop:semi-sat} Let $\beta, \delta>0$ be parameters such that $\beta^\delta <17^{-2}/2$. With $d=d(\beta,\delta)$ from Lemma \ref{lemma:inverse:semi} we have
  $$\P(\CF_{n-k,k_0, semi-sat}) \le e^{-n}.$$
\end{prop}
In particularly, with $\CE_{n-k, semi-sat}$ being the event $\wedge_{k\le k_0\le \eta n}\overline{\CF}_{n-k, k_0, semi-sat}$ in  the $\sigma$-algebra generated by $X_{k+1},\dots, X_n$,  then 
\begin{equation}\label{eqn:semi-sat}
  \P(\CE_{n-k, semi-sat}) \ge 1- e^{-n/2}.
\end{equation}

\begin{proof}[Proof of Proposition \ref{prop:semi-sat}] We have
  $$\P(\CF_{n-k, k_0, semi-sat}) = \sum_{V semi-sat} \P(W_{n-k} = V) \le  \sum_{V semi-sat} \P(X_{1},\dots, X_{n-k} \in V) .$$
  Now for each fixed $V$ that is semi-saturated of co-dimension $k_0\ge k$, by definition $\P(X\in V) \le 17 q^{-k_0}$. So 
  \[
    \P(X_{1},\dots, X_{n-k} \in V) \le 17^{n-k} q^{-k_0(n-k)}.
  \]
  
  We next use Lemma \ref{lemma:inverse:semi} to count the number $N_{semi-sat}$ of semi-saturated subspaces $V$. 
  Each $V$ is determined by its annihilator $V^\perp$ (of cardinality $q^{k_0}$). 
  To count $V^\perp$, we first choose a vector $v \in \CR$, 
  and then another $(k_0-1)$ dimensional subspace that is linearly independent of $v.$ 
  The number of ways to complete this space equals the number of ways to pick a $(k_0-1)$--dimensional space from
  $\F_q^{n-1}.$
  The number of such subspaces is given by the well--known (see \cite[Proposition 1.3.18]{Stanley}) exact formula
  \[
    \prod_{j=0}^{k_0-2} \frac{q^{n-1-j}-1}{q^{k_0-1-j}-1}
    \leq C q^{(k_0-1)(n-k_0)},
  \]
  for some absolute constant $C>0.$  Therefore
  \[
    N_{semi-sat} \leq C (2\beta^\delta)^n q^{nk_0-k_0^2+k_0}.
  \]

  Therefore
  \begin{align*}
    \P(\CF_{k, k_0, semi-sat}) & \le  \sum_{V semi-sat} \P(X_{1},\dots, X_{n-k} \in V) =O\Big( (2\beta^\delta)^n q^{nk_0 -k_0^2+k_0} 17^{n-k} q^{-k_0(n-k)} \Big)\\
    & = O\Big( 17^{n-k} (2\beta^\delta)^n q^{k_0} q^{k_0(k-k_0)} \Big)= O\Big( 17^{n-k} (2\beta^\delta)^n q^{k_0} \Big).
  \end{align*}
  Now recall that $e^{-d \alpha_n n} \le 16 q^{-k_0}$, and so 
  $$\P(\CF_{n-k, k_0, semi-sat}) =O(17^{n-k} (2\beta^\delta)^n q^{k_0}) = O(17^{n+1-k} (2\beta^\delta)^n e^{d \alpha_n n}).$$
  We then choose $\beta$ so that $2\beta^\delta <17^{-2}$ and with $d<1$ we have $\P(\CF_{n-k, k_0, semi-sat}) \le e^{-n}$.
\end{proof}

\subsection{Unsaturated subspace}\label{sub:unsat}  Recall that $k\le \eta n$ for sufficiently small $\eta$. Let $V$ be a subspace of codimension $k_0 \ge  k$ in $\F_q^n$. We say that $V$ is {\it unsaturated} if $V$ is not $\delta$-sparse and 
$$\max( e^{-d \alpha_n n}, 16 q^{-k_0}) < |\P(X \in V) -q^{-k_0}|.$$
In particularly this implies that 
$$\P(X \in V)\ge \max\{17 q^{-k_0} , \frac{16}{17} e^{-d \alpha_n n}\}.$$

The following is from \cite[Lemma 2.8]{M1}.

\begin{lemma}\label{lemma:swapping}
  There is an $\alpha_n'$-balanced probability distribution $\nu$ on $\F_q$ with $\alpha_n' = \alpha_n/64$ such that if $Y \in \F_q^n$ is a random vector with iid coefficients distributed according to $\nu$, then for any unsaturated proper subspace $V$
  $$|\P(X\in V) - \frac{1}{q^{k_0}}| \le (\frac{1}{2}+o(1)) |\P(Y\in V) - \frac{1}{q^{k_0}}| .$$ 
\end{lemma}

A proof of this lemma will be given in {\bf Part II}. By definition, if $V$ is unsaturated then
$$\P(Y \in V) \ge (2-o(1)) (\P(X\in V) -\frac{1}{q^{k_0}} ) +\frac{1}{q^{k_0}} > \frac{3}{2} \P(X\in V).$$

\begin{definition} Let $V$ be a subspace in $\F_q^n$. Let $d_{comb}\in \{1/n, \dots, n^2/n\}$. We say that $V$ has combinatorial codimension $d_{comb}$ if 
  $$(1- \alpha_n)^{d_{comb}} \le \P(X\in V) \le (1-\alpha_n)^{d_{comb}-1/n}.$$
\end{definition}
Now as we are in the unsaturated case, $\P(X \in V)\ge  \frac{16}{17} e^{-d \alpha_n n}$, and so
$$d_{comb} \le  2 d n.$$
In what follows we will fix $d_{comb}$ from the above range, noting that $d$ is sufficiently small, and there are only $O(n^2)$ choices of $d_{comb}$.

Let be fixed any $0<\delta_1<\delta_2 <1/3$ such that 
\begin{equation}\label{eq:funnydelta}
  16(\delta_2-\delta_1) (1+\log \frac{1}{\delta_2-\delta_1}) < \delta_1.
\end{equation}
Set 
$$r:=\lfloor \delta_1 n \rfloor \mbox{ and } s:=n-k-\lfloor \delta_2 n\rfloor.$$

Let $Y_1,\dots, Y_r$ be random vectors with entries distributed by $\nu$ obtained by Lemma \ref{lemma:swapping},
and let $Z_1,\dots, Z_s$ and $X_{r+s+1},\dots, X_{n-k}$ be random vectors with entries distributed by $\mu$. 

\begin{prop}\label{prop:unsat} 
  Let $W = \Sp\left\{ X_{1},\dots, X_{n-k} \right\}.$
  We have
  \[
    \P\Big( r+s \leq \dim(W) \leq n-k, W \mbox{unsaturated}\Big) 
    \le (3/2)^{-r/2} \binom{n-k}{r+s} \le (3/2)^{-\delta_1 n/4}.
  \]
\end{prop}
The second inequality follows directly from \eqref{eq:funnydelta} and the standard bound $\binom{n}{k} \leq \left( \frac{en}{n-k} \right)^{n-k}$. Notice that we do not require $\{X_i\}$ to be linearly independent. 
In other words, let $\CE_{n-k, unsat}$ denote the complement of the event above in the $\sigma$-algebra generated by $X_{1},\dots, X_{n-k}$, then
\begin{equation}\label{eqn:unsat}
  \P(\CE_{n-k, unsat}) \ge 1-(3/2)^{-\delta_1 n/4}.
\end{equation}
To prove Proposition \ref{prop:unsat} we show the following:

\begin{theorem}\label{theorem:swapXY} Let $V$ be any subspace of dimension between $r+s$ and $n-k$ and having $d_{comb} \le 2d n$. Then we have
  \[
    \P\Big( \Sp \{X_{1},\dots, X_{n-k}\} = V\Big) \le (3/2)^{-r/2} \binom{n-k}{r+s}  
    \P\Big(
    \Sp\left\{ 
      \left\{ Y_i \right\}_1^r
      ,
      \left\{ Z_i \right\}_1^s
      ,
      \left\{ X_i \right\}_{r+s+1}^{n-k}
    \right\}=V
    \Big).
  \]
\end{theorem}
To conclude Proposition \ref{prop:unsat} we then just use  
\begin{equation}\label{eqn:unsat:id}
  \sum_{\substack{V \le \F_q^n \\ \codim(V) \ge k}}
     \P\Big(
    \Sp\left\{ 
      \left\{ Y_i \right\}_1^r
      ,
      \left\{ Z_i \right\}_1^s
      ,
      \left\{ X_i \right\}_{r+s+1}^{n-k}
    \right\}=V
    \Big)
    = 1.
\end{equation}

\begin{proof}[Proof of Theorem \ref{theorem:swapXY}] 
  First of all, by independence between $X_i,Y_j, Z_l$,
  \begin{align}\label{eqn:unsat:1}
    &\P\Big(\Sp\{X_{1},\dots, X_{n-k}\} = V\Big) \times \P\Big(Y_1,\dots, Y_r, Z_1,\dots, Z_s \mbox{ linearly independent in } V\Big) \nonumber \\
    =&\P\Big(\Sp\{X_{1},\dots, X_{n-k}\} =V \wedge Y_1,\dots, Y_r, Z_1,\dots, Z_s \mbox{ linearly independent in } V\Big).
  \end{align}
  We next estimate $\P(Y_1,\dots, Y_r, Z_1,\dots, Z_s \mbox{ linearly independent in } V)$. By conditioning,
  \begin{align*}
    &\P\Big(Z_1,\dots, Z_s, Y_1,\dots, Y_r \mbox{ linearly independent in } V\Big) \\
    =&\P\Big(Y_r \in V\Big) \P\Big(Y_{r-1} \in V, Y_{r-1} \notin \lang Y_r \rang | Y_r \in V\Big) \cdots  \P\Big(Y_1 \in V, Y_1 \notin \lang Y_2,\dots, Y_r \rang | Y_2,\dots, Y_r \mbox{ lin. in $V$}\Big) \\
    \times & \cdots \times \P\Big (Z_1 \in V, Z_1 \notin \lang Z_2,\dots, Z_r, Y_1,\dots, Y_r \rang | Z_2,\dots, Z_s, Y_1,\dots, Y_r \mbox{ lin. in $V$} \Big).
  \end{align*}
  We first estimate the terms involving $Y_i$. By Lemma \ref{lemma:O}
  \begin{align*}
    & \P\Big(Y_i \in V, Y_i \notin \lang Y_{i+1},\dots, Y_r \rang | Y_{i+1},\dots Y_r \mbox{ lin. in $V$} \Big)  
    \ge \P(Y_i\in V) - (1-\alpha')^{n-(r-i)} \\ 
    \ge &\frac{3}{2}\P(X_i\in V) - (1-\alpha_n')^{n-(r-i)} \ge  \frac{3}{2}(1- \alpha_n)^{d_{comb}}   - (1-\alpha_n')^{n-(r-i)}\\ 
    \ge &\frac{3}{2}(1-\alpha_n)^{d_{comb}} (1- (1-\alpha_n)^{n/256 -d_{comb}}),
  \end{align*}
  where we used that $\alpha_n' =\alpha_n/64$ and $n-r  \ge (1-\delta_1) n \ge n/2$.

  Similarly, the terms involving $Z_i$ can be estimated as
  \begin{align*}
    &\P\Big(Z_i \in V, Z_i \notin \lang Z_{i+1},\dots, Z_s, Y_1,\dots, Y_r \rang |  Z_{i+1},\dots, Z_s, Y_1,\dots, Y_r \Big) \\
    \ge &\P(Z_i\in V) - (1-\alpha_n')^{n-(r+s-i)} \\
    \ge &(1- \alpha_n)^{d_{comb}}   - (1-\alpha_n')^{n-(r+s-i)} \ge (1-\alpha_n)^{d_{comb}}- (1-\alpha_n)^{n/256},
  \end{align*}
  where we used that $r+s = n-k - (\lfloor \delta_2 n \rfloor - \lfloor \delta_1 n \rfloor) \ge n/2$. 
  So 
  \begin{align}\label{eqn:unsat:2}
    \P\Big (Y_1,\dots, Y_r, Z_1,\dots, Z_s \mbox{ linearly independent in } V \Big) &\ge (3/2)^r (1-\alpha_n)^{(r+s) d_{comb}} \Big(1- (1-\alpha_n)^{n/256 -d_{comb}}\Big)^{r+s} \nonumber \\
    &\ge  (3/2)^{r-1} (1-\alpha_n)^{(r+s) d_{comb}},
  \end{align}
  where we used $d_{comb} \le 2 dn$ and $d$ is sufficiently small.

  Now we estimate the probability $\P(\Sp\{X_{1},\dots, X_{n-k}\}=V \wedge Y_1,\dots, Y_r, Z_1,\dots, Z_s \mbox{ linearly independent in } V)$ in~\eqref{eqn:unsat:1}. Since $Y_1,\dots, Y_r, Z_1,\dots, Z_s$ are linearly independent in $V$ and $\Sp\{X_{1},\dots, X_{n-k}\}=V$, there exists $n-k-r-s$ vectors $X_{i_1},\dots, X_{i_{n-k-r-s}}$ which together with $Y_1,\dots,Y_r, Z_1,\dots, Z_s$ are a basis for $V$. With a loss of a factor $\binom{n-k}{r+s}$ in probability, we can assume that 
  \(
  \Sp\left\{ 
      \left\{ Y_i \right\}_1^r
      ,
      \left\{ Z_i \right\}_1^s
      ,
      \left\{ X_i \right\}_{r+s+1}^{n-k}
    \right\}=V,
    \)
  and the remaining vectors $X_{1},\dots, X_{r+s}$ belong to $V$. Thus,
  \begin{align}\label{eqn:unsat:3}
    &\P\Big(\Sp\{X_{1},\dots, X_{n-k}\}=V \wedge Y_1,\dots, Y_r, Z_1,\dots, Z_s \mbox{ linearly independent in } V\Big) \nonumber \\
    \le & \binom{n-k}{r+s} \P\Big(\Sp\left\{ 
      \left\{ Y_i \right\}_1^r
      ,
      \left\{ Z_i \right\}_1^s
      ,
      \left\{ X_i \right\}_{r+s+1}^{n-k}
    \right\}=V \wedge X_{1},\dots, X_{r+s}\in V\Big)\nonumber \\
    \le & \binom{n-k}{r+s} \P\Big(\Sp\left\{ 
      \left\{ Y_i \right\}_1^r
      ,
      \left\{ Z_i \right\}_1^s
      ,
      \left\{ X_i \right\}_{r+s+1}^{n-k}
    \right\}=V\Big) \P(X_{1},\dots, X_{r+s}\in V)\nonumber \\
    \le & \binom{n-k}{r+s} \P\Big(\Sp\left\{ 
      \left\{ Y_i \right\}_1^r
      ,
      \left\{ Z_i \right\}_1^s
      ,
      \left\{ X_i \right\}_{r+s+1}^{n-k}
    \right\}=V\Big) (1-\alpha_n)^{(r+s) (d_{comb}-1/n)}.
  \end{align}
  Putting \eqref{eqn:unsat:1}, \eqref{eqn:unsat:2} and \eqref{eqn:unsat:3} together,  
  \begin{align*}\label{eqn:unsat:3}
    & \P\Big(\Sp\{X_{1},\dots, X_{n-k}\}=V\Big) \\
    =&\frac{\P\Big(\Sp\{X_{1},\dots, X_{n-k}\}=V \wedge Y_1,\dots, Y_r, Z_1,\dots, Z_s \mbox{ linearly independent in } V\Big)}{\P\Big(Y_1,\dots, Y_r, Z_1,\dots, Z_s \mbox{ linearly independent in } V\Big)} \nonumber \\
    \le & (3/2)^{-r+1} (1-\alpha_n)^{-(r+s) d_{comb}} \binom{n-k}{r+s} \P\Big(\Sp\left\{ 
      \left\{ Y_i \right\}_1^r
      ,
      \left\{ Z_i \right\}_1^s
      ,
      \left\{ X_i \right\}_{r+s+1}^{n-k}
    \right\}=V\Big) (1-\alpha_n)^{(r+s) (d_{comb}-1/n)} \nonumber  \\
    \le & (3/2)^{-r/2}  \binom{n-k}{r+s} \P\Big(\Sp\left\{ 
      \left\{ Y_i \right\}_1^r
      ,
      \left\{ Z_i \right\}_1^s
      ,
      \left\{ X_i \right\}_{r+s+1}^{n-k}
    \right\}=V\Big) .
  \end{align*}

\end{proof}

We remark that our proof above follows \cite[Section 4]{TV}. The treatment of \cite{M1} is similar but the author oversimplified the process by relying on the aforementioned incorrect result \cite[Proposition 2.3]{M1}. We now conclude the main result. 
\begin{proof}[Proof of Theorem \ref{theorem:corank:main}]
  Let $\CE_{n-k, dense}, \CE_{n-k, semi-sat}, \CE_{n-k, unsat}$ be the events introduced in \eqref{eqn:dense}, \eqref{eqn:semi-sat}, \eqref{eqn:unsat}. By definition, on these events, if $\codim(W_{n-k})=k_0$ then  
  $$|\P(X\in W_{n-k})-\frac{1}{q^{k_0}}| \le e^{-d \alpha_n n}.$$
\end{proof}

Next we give a proof of Theorem \ref{theorem:corank:main'}. The statement is clearly equivalent to Theorem \ref{theorem:corank:main} if $k> u$. In what follows we assume $k\le u$.
\begin{proof}[Proof of Theorem \ref{theorem:corank:main'}]. By Proposition \ref{prop:sparse}, with probability at least $1- e^{-c \alpha_n n}$ the subspace $\lang X_{1},\dots ,X_{n-k}\rang$ is not $\delta$-sparse, and hence $\lang X_{1},\dots ,X_{n+u-k}\rang$ is also not $\delta$-sparse on this event. 

  For the semi-saturated subspace, the conclusion of Proposition \ref{prop:semi-sat} continues to hold using Lemma \ref{lemma:inverse:semi}. Indeed, with the same choice of parameters, by the proof of Proposition of \ref{prop:semi-sat} 
  \begin{align*}
    \P(\CF_{n+u-k, k_0, semi-sat}) = \sum_{V semi-sat} \P(W_{n+u-k} = V) &\le  \sum_{V semi-sat} \P(X_{1},\dots, X_{n+u-k} \in V)\\ 
    &\le  \sum_{V semi-sat} \P(X_{1},\dots, X_{n-k} \in V) \le e^{-n}.
  \end{align*}

  Finally, for unsaturated subspaces, by the same method we can show the following analog of Proposition \ref{prop:unsat} with the same parameters
  \begin{equation}\label{eqn:unsat:n+u}
    \P\Big(\Sp\{X_{1},\dots, X_{n+u-k}\} \mbox{ unsaturated and of dim. between $r+s$ and $n$}\Big) \le (3/2)^{-r/2} \binom{n+u-k}{r+s}.
  \end{equation}
  Indeed, to justify this result we just use the same swapping method of Theorem \ref{theorem:swapXY}; the only difference is that there are $\binom{n+u-k}{r+s}$ ways to choose the $X_{1},\dots, X_{r+s}$ in \eqref{eqn:unsat:3}. The bound~\eqref{eqn:unsat:n+u} is again smaller than $(3/2)^{-\delta_1 n/4}$ if $u\le \eta n $ with $\eta$ sufficiently small compared to $\delta_1$. 
\end{proof}
\vskip .1in
{\bf \underline{Part II. Proofs of the lemmas}}\label{appendix:lemmas}: Here we sketch the proof of Lemma \ref{lemma:inverse:semi}, Lemma \ref{lemma:swapping}, and Theorem \ref{theorem:LO} by following \cite{M1}.

Recall that $\tr: \F_q \to \F_p$ is the field trace, which gives rise to the isomorphism between $\F_q$ and $\wh{\F}_q$ by $x \to e_p(\tr(tx)) = \exp(2\pi i \tr(tx)/p)$ (and so we will identify $\wh{\F}_q$ with $\F_q$). Let $\mu$ be a probability measure on $\F_q$. The Fourier transform of $\mu$ is
$$\wh{\mu}(x)=\sum_{t\in \F_q} \mu(t) e_p(\tr(x t)) = \E e_p(\tr(x \xi)), \mbox{ where $\xi$ has distribution $\mu$}.$$
We define the additive spectrum by 
$$\Spec_{1-\eps} \mu : =\{x \in \F_q, |\wh{\mu}(x)| \ge 1-\eps\}.$$

We will be using the following two important results. 

\begin{theorem}[Kneser]\cite[Theorem 5.5]{TVbook} Let $A,B \subset Z$ be finite subsets of an abelian group $Z$. Then
  $$|A+B| + |\operatorname{Sym}(A+B)| \ge |A|+|B|,$$
  where $\operatorname{Sym}(X) = \{h\in Z: h+X=X\}$. 
\end{theorem}
Note that $\operatorname{Sym}(X)$ is a symmetric additive subgroup of $Z$.  As $\operatorname{Sym}(A_1+\dots+A_k)$ is increasing in $k$, iterating the result we obtain
\begin{cor}\label{cor:Kneser}
  $$|A_1+\dots+A_k| +(k-1) |\operatorname{Sym}(A_1+\dots+A_k)| \ge |A_1|+\dots + |A_k|.$$
\end{cor}

To start with, we prove the following version of the classical Erd\H{o}s-Littlewood-Offord result in finite field, which was used in the proof of Proposition \ref{prop:sparse}. 
\begin{theorem}\cite[Theorem 2.4]{M1}\label{theorem:LO}
  Let $X \in \F_q^n$ be a random vector with iid entries taken from an $\alpha$-balanced distribution. Suppose $w \in \F_q^n$ has at least $m$ non-zero coefficients. Then we have
  $$|\P(X \cdot w =r) -\frac{1}{q}| \le \frac{2}{\sqrt{\alpha m}}.$$
\end{theorem}

\begin{proof}[Proof of Theorem \ref{theorem:LO}] Let $\xi_1,\dots, \xi_n$ denote the entries of $X$, and $w_1,\dots,w_n$ denote the components of $w$. 
  By $\F_p$--linearity of the field trace, we can write
  \[
    1_{X \cdot w=r}
    =
    \frac{1}{q} \sum_{t \in \F_q} e_p\Big(\tr(\sum_l \xi_l w_l)\Big)e_p\Big(-\tr(rt)\Big).
  \] 
  By independence of the $\{\xi_i\}_1^n$ we can therefore write
  \[
    \P(X \cdot w =r) = 
    \E (1_{X \cdot w=r}) = 
    \frac{1}{q} \sum_{t \in \F_q} \prod_{l=1}^n \E e_p(\tr(\xi_l w_l t)) e_p(-\tr(rt)).
  \]
  By the triangle inequality
  \[
    |\P(X \cdot w =r) -\frac{1}{q}| \le  
    \frac{1}{q} \sum_{t \in \F_q, t\neq 0} \prod_{l=1}^n |\E e_p(\tr(\xi_l w_l t))| =  
    \frac{1}{q} \sum_{t \in \F_q, t\neq 0} \prod_{l=1}^n |\wh{\mu}(w_lt)|.
  \]
  Define $\psi(t) =1 -|\wh{\mu}(t)|^2$. Using $|x| \le \exp(-(1-x^2)/2)$ for $|x|\le 1$, 
  \[
    |\P(X \cdot w =r) -\frac{1}{q}| \le  
    \frac{1}{q}  \sum_{t \in \F_q, t\neq 0}  \exp\biggl(-\frac{1}{2} \sum_{l=1}^n \psi(w_lt)\biggr).
  \]
  Set $f(t) = \sum_l \psi(w_l t)$, then
  \begin{equation}\label{eqn:level}
    |\P(X \cdot w =r) -\frac{1}{q}| \le  \frac{1}{2} \int_0^\infty  \frac{1}{q}  |\{t \neq 0, f(t) \le v\}| e^{-v/2} dv=  \frac{1}{2} \int_0^\infty  \frac{1}{q} |T'(v)|  e^{-v/2} dv
  \end{equation}
  where the level sets are defined as
  $$T(v):= \{t, f(t) \le v\} \mbox{ and } T'(v)= T(v)\bs \{0\}.$$
  \begin{claim}
    For any $v>0$ we have 
    $$kT(v) = T(v)+\dots +T(v) \subset T(k^2v).$$
  \end{claim}
  \begin{proof} It suffices to show that for any $\beta_1,\dots, \beta_k \in \F_q$ 
    $$\psi(\beta_1+\dots+\beta_k) \le k (\psi(\beta_1)+\dots+\psi(\beta_k)).$$
    Indeed, by definition of $\psi$, after squaring out the above is equivalent to
    $$1 -\sum_{t_1,t_2\in \F_q} \mu(t_1)\mu(-t_2) \cos\Big(\frac{2\pi}{p} \tr((t_1+t_2)(\beta_1+\dots+\beta_k))\Big) \le k^2 -k \sum_i \sum_{t_1,t_2\in \F_q} \mu(t_1)\mu(-t_2) \cos\Big(\frac{2\pi}{p} \tr((t_1+t_2)\beta_i)\Big).$$
      Hence it suffices to observe that for all real numbers
      $\left( \beta_i \right)_1^k,$
      $\cos(\beta_1+\dots +\beta_k) \ge k \sum \cos \beta_i -k^2+1,$ which we justify now.
      If for even a single $\beta_i,$ $\cos(\beta_i) \leq 0,$ then the largest value that can be attained on the right
      hand side is $1-k,$ from which it follows the equality holds.  
      Hence by periodicity, we may assume that all these $\beta_i \in (-\pi/2,\pi/2).$
      The function 
      \[
	\left( \beta_i \right)_1^k \mapsto 
	\cos(\beta_1+\dots +\beta_k)
	- k \sum \cos(\beta_i)
      \]
      is smooth and its only local minimum in the domain considered can be checked to occur at $0,$ at which value equality is attained.

  \end{proof}

  By Corollary \ref{cor:Kneser} 
  \begin{equation}\label{eqn:Kneser:v}
    k T(v) \le |T(k^2v)| + (k-1) \operatorname{Sym}(T(v)+\dots+ T(v)).
  \end{equation}

  We next claim that if $k^2 v <\alpha_n m$, that is $k < \sqrt{\frac{\alpha_n m}{v}}$, then $T(k^2v)$ contains no-nontrivial additive subgroup $H$ of $\F_q$, and so $|\operatorname{Sym}(T(v)+\dots+ T(v))|=1$. Indeed, fix a subgroup $H$, then
  $$|H|^{-1} \sum_{t\in H} f(t) = \sum_{l=1}^n |H|^{-1} \sum_{t\in H} \psi(w_l t) = \sum_{l=1}^n |H|^{-1} \sum_{t\in H} (1- |\wh{\mu}(w_l t)|^2).$$
  By the Fourier inversion formula, and by the $\alpha$-balanced assumption on the distribution $\mu$
  $$ |H|^{-1} \sum_{t\in H} (1- |\wh{\mu}(w_l t)|^2) = \sum_{s_1, s_2 \in \F_q} \mu(s_1)\mu(s_2) 1_{H^\perp} (w_l(s_1-s_2)) \le 1-\alpha_n.$$
  Since at least $m$ of the choices $w_l$ are non-zero
  $$|H|^{-1} \sum_t f(t) \ge \alpha_n m.$$
  Thus there exists $t\in H$ such that $f(t) \ge \alpha_n m$, and so $H \not \subset T(k^2 v)$. So by \eqref{eqn:Kneser:v} we have 
  $$|T'(v)| \le \sqrt{\frac{v}{\alpha_n m}}|T'(\alpha_n m)| \le  \sqrt{\frac{v}{\alpha_n m}} q, \mbox{ for all $v \le \alpha_n m$}.$$
  Substitute back to \eqref{eqn:level} we obtain
  $$|\P(X \cdot w =r) -\frac{1}{q}| \le  \frac{1}{2} \frac{1}{\sqrt{\alpha_n m}} \int_0^\infty  \sqrt{v}e^{-v/2} dv + e^{-\alpha_n m/2}.$$
\end{proof}

Now we prove the other lemmas that were used in the treatment of semi-saturated and un-saturated subspaces.

\begin{proof}[Proof of Lemma \ref{lemma:inverse:semi}] 
  Note that $V$ is not $\delta$-sparse. Let $k_0=\codim(V)$, let $\xi_1,\dots, \xi_n$ denote the entries of $X$, we have

  $$\P(X\in V) =  \E  q^{-k_0} \sum_{t \in V^\perp} e_p\big(\tr(\sum_l \xi_l t_l)\big)  = q^{-k_0}\sum_{t \in V^\perp} \prod_{l=1}^n\E e_p(\tr(\xi_l t_l))= q^{-k_0} \sum_{t \in V^\perp} \prod_{l=1}^n \wh{\mu}(t_l),$$
  where $t_1,\dots,t_n$ denote the entries of $t$, and where we used the fact that $\sum_{t \in V^\perp} e_p(\tr(\sum_l \xi_l t_l)) =0$ if and only if $X \notin V$. By the triangle inequality,
  $$|\P(X\in V) -q^{-k_0}|\le q^{-k_0} \sum_{t \in V^\perp, t \neq 0} \prod_{l=1}^n |\wh{\mu}(t_l)|.$$
  By \corE{the} pigeonhole principle, for some $t \in V^\perp, t \neq 0$,
  \[
    e^{-d \alpha_n n} \le |\P(X\in V) -q^{-k_0}| \le   \prod_{l=1}^n |\wh{\mu}(t_l)|.
  \]
  Again, using $|x| \le \exp(-\frac{1}{2}(1-x^2))$ for $|x|\le 1$, 
  \[
    \sum_{l=1}^n 1 -|\wh{\mu}(t_l)|^2 \le 2d \alpha_n n.
  \]
  By averaging, there exists an index set $\sigma \subset [n]$ with $|\sigma| \ge (1 -\delta/2) n$ and $|\wh{\mu}(t_l)| \ge 1-10d  \delta^{-1} \a_n$ for $l\in \sigma$. In other words, for all $l\in \sigma$
  \[
    t_l \in \Spec_{1-10d  \delta^{-1} \a_n}\mu.
  \]
  \begin{claim}\label{claim:spec} The set $\Spec_{1-\alpha_n/2}$ does not contain any non-trivial additive subgroup $H$ of $\F_q$.
  \end{claim}
  \begin{proof} By Fourier's inversion formula
    $$(1-\alpha_n/2)^2 |H \cap \Spec_{1-\alpha_n/2}(\mu)| \le \sum_{s\in H} |\wh{\mu}(s)|^2 \le |H| (1-\alpha_n),$$
    where in the last estimate we used the fact that $\mu$ is $\alpha_n$-balanced.
  \end{proof}
  Set $k:= \lfloor \beta^{-1} \rfloor$ and choose $d$ so that 
  \begin{equation}\label{eqn:d}
    d \le k^{-2}  \delta/5.
  \end{equation}
  Then  
  \[
    \Spec_{1-10d \delta^{-1} \a_n} \subset \Spec_{1- 2k^{-2} \a_n}(\mu):=A.
  \]
  We next claim that 
  \begin{equation}\label{eqn:beta}
    |A\backslash \{0\}| \le \beta q.
  \end{equation}
  Indeed, this is because by applying Cauchy-Schwarz
  $$kA \subset \Spec_{1-2\a_n}(\mu).$$
  Furthermore,
  $$\operatorname{Sym}(kA) \subset (kA)-(kA) = 2kA \subset \Spec_{1-\a_n/2}(\mu).$$
  By Claim \ref{claim:spec}, $\operatorname{Sym}(kA)$ is trivially $\{0\}$. So by Corollary \ref{cor:Kneser}
  $$k|A| \le |kA| + (k-1) \le q + (k-1),$$
  proving \eqref{eqn:beta}.

  Finally, let $\CR$ be the set of non-zero $t$ in $\F_q^n$ which have at least $(1-\delta/2)n$ components $t_l$ in $\Spec_{1-10d\delta^{-1}\a}$. By \eqref{eqn:beta}, as $t \perp V$ and $V$ is not $\delta$-sparse, at least $\delta/2$ of the components $\xi_l$ are non-zero 
  $$|\CR| \le 2^n (\beta q)^{\delta n/2} q^{n-\delta n/2} \le (2 \beta^{\delta/2} q)^n.$$ 

\end{proof}

To prove Lemma \ref{lemma:swapping}, we use the following rather standard Hal\'asz--type construction from \cite[Proposition 3.6]{M1} (see also \cite[Lemma 7.1]{TV} or \cite{Halasz}).

\begin{prop}\label{prop:swapping'} There is a probability distribution $\nu: \F_q \to [0,1]$ depending on $\mu$ and $\alpha$ such that the following properties hold with $t=(t_1,\dots,t_n)$ and
  $$f(t) := \prod_{l=1}^n |\wh{\mu}(t_l)|,  g(t):= \prod_{l=1}^n |\wh{\nu}(t_l)|.$$
  \begin{itemize}
    \item For all $0<u<1$ we have $4F(u) \subset G(u)$, where 
      $$F(u) = \{t \in \F_q^n, |f(t)| > u\} \mbox{ and } G(u) = \{t \in \F_q^n, |g(t)| > u\}.$$
    \item For all $t \in V^\perp$, 
      $$f(t)\le g^{16}(t).$$
    \item  For all $t$,
      $$\wh{\nu}(t) \ge 0.$$
    \item $\nu$ is $\alpha'_n$-balanced with $\alpha_n'=\alpha_n/64$. 
  \end{itemize}
\end{prop}


\begin{proof}[Proof of Lemma \ref{lemma:swapping}] Note that 
  $$\P(Y\in V) -q^{-k} = q^{-k} \sum_{t\in V^\perp, t \neq 0} \prod_{l=1}^n \wh{\nu}(t_l) = q^{-k} \sum_{t\in V^\perp, t \neq 0} g(t).$$
  Thus to show $|\P(X\in V) -q^{-k}| \le (\frac{1}{2}+o(1)) |\P(Y\in V) -q^{-k}|$, as $\wh{\nu}\ge 0$ it suffices to show that 
  $$\sum_{t\in V^\perp, t \neq 0} f(t) \le (\frac{1}{2}+o(1)) \sum_{t\in V^\perp, t \neq 0} g(t).$$

  Let  $\eps>0$ be a parameter to be sent to 0. We write 
  $$ \sum_{t\in V^\perp, t \neq 0} f(t) = \sum_{t\in V^\perp, t \neq 0, f(t)< \eps} f(t)  + \sum_{t\in V^\perp, t \neq 0, f(t) \ge \eps} f(t) := \sum_{< \eps} (f) + \sum_{ \ge \eps} (f).$$
  As $f(t) \le g(t)^{16}$, we have 
  $$\sum_{< \eps} (f)  \le \eps^{15/16} \sum_{< \eps} (g) < (\frac{1}{2}+o(1))  \sum_{< \eps} (g).$$
  We also write 
  $$ \sum_{ \ge \eps} (f) = \int_{\eps}^\infty |F'(u)|du + \eps |F'(\eps)|,$$
  where 
  $$F'(u) = F(u) \backslash \{0\} =  \{t \in V^\perp, t \neq 0, |f(t)| > u\}.$$

  \begin{claim}\label{claim:trivial:H} With $\eps = \exp(-\frac{1}{2} \alpha_n' \delta n)$, the set $G(\eps)$ does not contain any non-trivial additive subgroup $H \le V^\perp$. In particularly, as $G(u)$ is decreasing, the same happens for any $u\ge \eps$.
  \end{claim}
  \begin{proof}
    Clearly we can assume $H \cong \Z/p\Z$. Assume that $w=(w_1,\dots, w_n)\in V^\perp$ that generates $H$. Since $V$ is unsaturated (and hence not $\delta$-sparse), $w$ has at least $\delta n$ non-zero components. 
    Define $h(t):= \sum_{l=1}^n 1 - |\wh{\nu}(t_l)|^2, t\in H$. We can also write $h(t):= \sum_{i=1}^k 1 - |\wh{\nu}(t w_{l_i})|^2, t\in \Z/p\Z$, where $w_{l_i}$ are non-zero. 
    By Fourier's inversion formula, and as $\nu$ is $\alpha_n'$-balanced, $\sum_{t\in \Z/p\Z} |\wh{\nu}(t w_{l_i})|^2 \le p (1-\alpha_n')$. So
    $$\sum_{t\in H} h(t) \ge  p \alpha_n' k \ge p \alpha_n' \delta n.$$
    By pigeon-hole principle, there exists $t\in H$ such that $h(t) \ge \alpha_n' \delta n$. On the other hand
    $$g(t) =\prod_{l=1}^n |\wh{\nu}(t_l)| \le \exp(-\frac{1}{2} \sum_{l=1}^n 1- |\wh{\nu}(t_l)|^2) = \exp(-\frac{1}{2} h(t)) \le \exp(-\frac{1}{2} \alpha_n' \delta n) =\eps .$$ 
    We thus have found an element $t\in H$ which lies outside $G(\eps)$. 
  \end{proof}
  Let $u\ge \eps$. By Proposition \ref{prop:swapping'}, $\operatorname{Sym}(2F(u)) \subset 4F(u) \subset G(u)$. Thus by Claim \ref{claim:trivial:H}, the additive subgroup $\operatorname{Sym}(2F(u))$ must be trivial, and so by Corollary \ref{cor:Kneser} 
  $$2 |F(u)| \le |\operatorname{Sym}(F(u)+F(u))| + |F(u)+F(u)| \le 1 + |G(u)|.$$
  It thus follows that for all $u\ge \eps$ we have $2|F'(u)| \le |G'(u)|$. So
  $$\int_{\eps}^\infty |F'(u)|du + \eps |F'(\eps)| \le \frac{1}{2} (\int_{\eps}^\infty |G'(u)| du+ \eps |G'(\eps)|),$$
  completing the proof of Lemma \ref{lemma:swapping}.
\end{proof}


\end{document}